\documentclass[12pt]{amsart}
\pdfoutput=1
\usepackage{amsthm,amsmath,amssymb,tikz-cd}
\usepackage{fullpage}
\newtheorem{theorem}{Theorem}
\newtheorem{lemma}[theorem]{Lemma}
\newtheorem{proposition}[theorem]{Proposition}
\newtheorem{corollary}[theorem]{Corollary}
\theoremstyle{definition}
\newtheorem*{notation}{Notation}
\newcommand{\cat}{\mathbf}

\newcommand{\inprod}[2]{\ensuremath{\langle #1\,|\,#2 \rangle}}
\newcommand{\inv}{\vphantom{f}^{-1}}
\newcommand{\id}{\mathrm{id}}
\renewcommand{\:}{\colon}
\newcommand{\NN}{\mathbb N}

\begin{document}
\title{Axioms for the category of Hilbert spaces}
\author{Chris Heunen}
\address{University of Edinburgh, School of Informatics, 10 Crichton Street, EH8 9AB, Edinburgh, United Kingdom}
\email{chris.heunen@ed.ac.uk}
\author{Andre Kornell}
\address{Tulane University, Department of Computer Science, 
6823 St.\ Charles Avenue, New Orleans, LA 70118, United States}
\email{akornell@tulane.edu}
\thanks{Chris Heunen was supported by EPSRC Fellowship EP/R044759/1. Andre Kornell was supported by AFOSR MURI grant FA9550-16-1-0082.}
\begin{abstract}
	We provide axioms that guarantee a category is equivalent to that of continuous linear functions between Hilbert spaces. 
	The axioms are purely categorical and do not presuppose any analytical structure.
	This addresses a question about the mathematical foundations of quantum theory raised in reconstruction programmes such as those of von Neumann, Mackey, Jauch, Piron, Abramsky, and Coecke.
\end{abstract}
\maketitle

Quantum mechanics has mathematically been firmly founded on Hilbert spaces and operators between them for nearly a century~\cite{vonneumann:foundations}. 
There has been continuous inquiry into the special status of this foundation since~\cite{redei:vonneumannhilbert,hardy:fiveaxioms,kapustin:exact}.
How are the mathematical axioms to be interpreted physically? 
Can the theory be reconstructed from a different framework whose axioms can be interpreted physically? 
Such reconstruction programmes involve a mathematical reformulation of (a generalisation of) the theory of Hilbert spaces and their operators, such as operator algebras~\cite{murrayneumann:ringsofoperators}, orthomodular lattices~\cite{jauch:foundations,piron:foundations}, and, most recently, categorical quantum mechanics~\cite{abramskycoecke:protocols,ghezlimaroberts:wstarcategories}.
The latter uses the framework of category theory~\cite{maclane:categories}, and emphasises operators more than their underlying Hilbert spaces. It postulates a category with structure that models physical features of quantum theory~\cite{heunenvicary}. 
The question of how ``to justify the use of Hilbert space''~\cite{piron:foundations} then becomes: which axioms guarantee that a category is equivalent to that of continuous linear functions between Hilbert spaces?
This article answers that mathematical question. 
The axioms are purely categorical in nature, and do not presuppose any analytical structure such as continuity, complex numbers, or probabilities.
The approach is similar to Lawvere's categorical characterisation of the theory of sets~\cite{lawvere:etcs}.

\section*{Axioms}


We consider six axioms on a (locally small) category $\cat{C}$.
The first two impose additional structure: an involution called a dagger and a tensor product.
All axioms have to respect the dagger. In particular, the right notion of inclusion is dagger monomorphism, and this notion permeates the last four axioms.
Axioms three and four demand finite (co)completeness; roughly, direct sums and equalisers.
The last two axioms ask that dagger subobjects behave well: intuitively, one should be able to divide by them and take directed suprema of them.
More precisely:
\begin{enumerate}
  \item[(D)]\label{axiom:dagger}	
    The category is equipped with a \emph{dagger} $\dag$.
    This is an operation that maps each morphism $f \colon H \to K$ to a morphism $f^\dagger \colon K \to H$ such that $f^{\dag\dag}=f$, and $(g \circ f)^\dag = f^\dag \circ g^\dag$, and $\id_H^\dag=\id_H$.
    If $f^\dag \circ f = \id_H$ we call $f$ a \emph{dagger monomorphism}, and if additionally $f \circ f^\dag = \id_K$ we call it a \emph{dagger isomorphism}.
    
  \item[(T)]\label{axiom:tensor}
    The category is equipped with a \emph{dagger monoidal} structure $\otimes$ whose unit $I$ is a \emph{simple monoidal separator}.
    The former means that the isomorphisms $(H \otimes K) \otimes L \simeq H \otimes (K \otimes L)$ and $I \otimes H \simeq H \simeq H \otimes I$ are dagger isomorphisms.
    The latter means that $I$ has exactly two subobjects, and that $f=g : H \otimes K \to L$ as soon as $f \circ (h \otimes k) = g \circ (h \otimes k)$ for all $h : I \to H$ and $k : I \to K$.

  \item[(B)]\label{axiom:biproducts}
  	Any two objects $H$ and $K$ have a \emph{dagger biproduct} $H \oplus K$, and there is a zero object $0$.
  	The latter means that for every object $H$ there exist unique morphisms $0 \to H$ and $H \to 0$. Hence any two objects $H$ and $K$ have a unique morphism $0_{H,K} : H \to K$ that factors through $0$.
  	The former means that there are dagger monomorphisms $i : H \to H \oplus K$ and $j : K \to H \oplus K$ that make $H \oplus K$ a coproduct of $H$ and $K$ and satisfy $j^\dag \circ i = 0_{H,K}$~\cite{heunenkarvonen:limitsindaggercategories}.

  \item[(E)]\label{axiom:equalisers}
    Any two morphisms $H \to K$ have a \emph{dagger equaliser} $f \colon E \to H$, that is, an equaliser $f$ that is a dagger monomorphism.

  \item[(K)]\label{axiom:kernels}
    Any \emph{dagger monomorphism} $N \to H$ is a dagger equaliser of some morphism $H \to K$ and $0_{H,K} : H \to K$.

  \item[(C)]\label{axiom:colimits}
  	The wide subcategory of dagger monomorphisms has \emph{directed colimits}. This means that if $(N, \leq)$ is a directed partially ordered set, $H_n$ are objects, and $H_{m \leq n} : {H_m \to H_n}$ are dagger monomorphisms, 
 	  then there is a universal object $H$ with dagger monomorphisms $H_n \to H$ that commute with the morphisms $H_{m \leq n}$.
\end{enumerate}

The categories $\cat{Hilb}_{\mathbb{R}}$ of real Hilbert spaces and $\cat{Hilb}_{\mathbb{C}}$ of complex Hilbert spaces with continuous linear functions satisfy these axioms: (D) is given by adjoints, (T) by tensor product, (B) by direct sum, (E) and (K) by closed subspaces, and (C) by the completion of the directed union.
We will show that any category $\cat{C}$ that satisfies these axioms is equivalent to one of these two categories.

\section*{Context}

This characterisation combines established research programmes on dagger categories and on orthomodular spaces. Dagger categories are categories with a certain self-duality and were first considered by MacLane~\cite{maclane:duality}; the terminology (D) was introduced by Selinger~\cite{selinger:cpm}. One research thread has been to find conditions that guarantee that the scalars of a monoidal dagger category form a field in the manner of Mitchell's embedding theorem for abelian categories~\cite{mitchell:embedding}. Abramsky and Coecke considered parts of (T) and (B) in observing that the scalars of a monoidal category with biproducts form a commutative semiring~\cite{abramskycoecke:protocols}. Vicary emphasised the compatibility of the dagger (D) with all other structure, added simplicity to (T), and introduced (E) to show that the scalars embed into a field~\cite{vicary:complex}. Heunen showed that adding (K) implies that the scalars are a field~\cite{heunen:embedding}. Vicary and Heunen both also embedded the scalars into $\mathbb{C}$ using additional assumptions that are not needed here.
Axioms (D)--(K) together imply that each object may be regarded as a Hermitian space.

The new axiom (C) implies that each such Hermitian space $\mathcal H$ is orthomodular. Orthomodular spaces were first studied by Kaplansky~\cite{kaplansky:forms}; the main research thread has been to find conditions that guarantee that an infinite-dimensional orthomodular space is a Hilbert space. Amemiya and Araki proved Piron's conjecture that a (real or complex) inner product space is an orthomodular space if and only if it is a Hilbert space~\cite{piron:conjecture,amemiyaaraki:piron}. Wilbur found algebraic conditions that imply that an orthomodular space is a Hilbert space~\cite{wilbur:standardquantumlogic}. Continuing research in this direction culminated in Sol{\`e}r's Theorem that any orthomodular space with an infinite orthonormal subset is a Hilbert space~\cite{gross:hilbertlattices, soler}. 
Sol{\`e}r's Theorem has been applied to characterise quantum logic as an orthomodular lattice~\cite{soler,maedamaeda,mayet:automorphisms,vetterlein:orthogonalityspaces}.
This research line about orthomodular lattices is of wide interest due to the reconstruction programme, which aims to derive the structure of Hilbert space from physical first principles. 
See~\cite{grinbaum:reconstruction} for a  literature survey of such reconstructions. 

\section*{Scalars}

We begin by looking at the morphisms $z \colon I \to I$. These are called \emph{scalars}, because in any category they form a monoid under composition, with $\id_I \colon I \to I$ as multiplicative unit.
In a monoidal category, as in axiom (T), this monoid is commutative, because the multiplication of two scalars $z,w \colon I \to I$ equals
\begin{equation}\label{eq:scalarscalarmultiplication}
  I \simeq I \otimes I \stackrel{w \otimes z}{\longrightarrow} I \otimes I \simeq I
\end{equation}
for which an Eckmann-Hilton argument holds~\cite[Lemma~2.3]{heunenvicary}.

The dagger of axiom (D) provides an involution on scalars.
The biproducts of axiom (B) affords a matrix calculus on morphisms~\cite[Section~2.2.4]{heunenvicary}. 
In particular, there are a codiagonal map $(i\,\, j) \colon I \oplus I \to I$ and a diagonal map $(i\,\,j)^\dag \colon I \to I \oplus I$ that let us add scalars $w,z \colon I \to I$:
\begin{equation}\label{eq:addition}
  I \stackrel{(i\,\,j)^\dag}{\longrightarrow} I \oplus I \stackrel{w \oplus z}{\longrightarrow} I \oplus I \stackrel{(i\,\, j)}{\longrightarrow} I.
\end{equation}
Together with the zero morphism $0_{I,I} \colon I \to I$, the scalars hence form an involutive commutative semiring~\cite[Lemma~2.15]{heunenvicary}.

\begin{lemma}\label{lem:field}
  $\mathcal{I}=\cat{C}(I,I)$ is a field with involution $z \mapsto z^\dagger$.
\end{lemma}
\begin{proof}[Proof (sketch)]
  To see that there are multiplicative inverses, let $z \in \mathcal{I}$. By axiom (E), $z$ factors as an epimorphism $I \to H$ followed by a dagger monomorphism $f \colon H \to I$. It then follows from the simplicity of $I$ in axiom (T) that $f$ is either zero or invertible. If $f=0$ then $z=0$. If $f$ is invertible then $z$ is an epimorphism. But then $z^\dag$ is monic and hence zero or invertible. In either case $z$ is zero or invertible.

  It follows that $\mathcal{I}$ must either have additive inverses or be zerosumfree. The latter term means that $w+z=0$ implies $w=z=0$. But this would contradict axiom (K), because the equaliser $\ker(f)$ of $f \colon H \to K$ and $0_{H,K}$ is zero if and only if $f$ is a monomorphism, but according to~\eqref{eq:addition} $\ker(i\,\, j)=0$ if $w+z=0$ whereas $(i\,\,j)\neq 0$.
  For details, see~\cite[Lemma~4.5]{heunen:embedding}. 
\end{proof}

We will see in Proposition~\ref{prop:hilbertspace} below that in fact $\mathcal{I}$ must be $\mathbb{R}$ (with the trivial involution) or $\mathbb{C}$. Furthermore, any morphism $f \colon H \to K$ can be multiplied by a scalar $z \colon I \to I$ to get:
\begin{equation}\label{eq:scalarmultiplication}
  H \simeq I \otimes H \stackrel{z \otimes f}{\longrightarrow} I \otimes K \simeq K.
\end{equation}
Similarly, the addition of~\eqref{eq:addition} works just as well for parallel morphisms $f,g \colon H \to K$ of arbitrary type. It follows from Lemma~\ref{lem:field} that the morphisms $H \to K$ always form a vector space over $\mathcal{I}$.
In particular, we may consider morphisms $I \to H$ to be \emph{vectors}.

\section*{Projections}

Now consider \emph{projections} on an object $H$, that is, morphisms $p \colon H \to H$ satisfying $p^\dag \circ p = p$. They are partially ordered by $p \leq q$ if and only if $q \circ p = p$.
We next show that arbitrary subsets of projections have least upper bounds, and there is an \emph{orthocomplement} $p^\perp$ satisfying $p^{\perp\perp}=p$, and $p \leq q$ if and only if $q^\perp \leq p^\perp$.

\begin{lemma}\label{lem:completeortholattice}
  The projections on an object $H$ form a complete ortholattice with $p^\perp = \id_H - p$.
\end{lemma}
\begin{proof}[Proof (sketch)]
  Projections on $H$ are in one-to-one correspondence with \emph{dagger subobjects} of $H$, that is, dagger isomorphism classes of dagger monomorphisms $f \colon N \to H$. These are also partially ordered, by $f \leq g$ when $f$ factors through $g$. Axiom (K) provides an order isomorphism between projections and dagger subobjects, 
  see~\cite[Lemma~5 and Proposition~12]{heunenjacobs:daggerkernelcategories}.
  By Axiom (C), these partially ordered sets furthermore have least upper bounds of directed subsets.

  The orthocomplement of a dagger subobject $f \colon N \to H$ is given by $\ker(f^\dag)$. Axioms (D) and (E) ensure that it indeed satisfies $f^{\perp\perp}=f$, and $f \leq g$ if and only if $g^\perp \leq f^\perp$~\cite[Lemma~1]{heunenjacobs:daggerkernelcategories}.

  By axioms (B) and (E), pullbacks exist, and so any two dagger subobjects have a greatest lower bound~\cite[Lemma~2]{heunenjacobs:daggerkernelcategories}. 
  But because of the orthocomplementation any two dagger subobjects then have a least upper bound. Hence any finite number of projections have a least upper bound.

  Now, the least upper bound of an arbitrary subset of projections is the least upper bound of the directed family of least upper bounds of finite subsets, and hence exists.
  Finally, it is clear that $p^\perp = \id_H - p$ is an orthocomplement.
\end{proof}

\begin{notation}
  Morphisms out of $I$ will play a dual role, as we will also regard them as vectors in a Hilbert space. To distinguish between these two roles we will use the following notation.
  For an object $H$, we write $\mathcal{H}$ for the homset $\cat{C}(I,H)$.
  For a morphism $f \colon H \to K$, we write $F \colon \mathcal{H} \to \mathcal{K}$ for the function that postcomposes with $f$.
\end{notation}

We have already established that $\mathcal{H}$ is a vector space (over $\mathcal{I}$), and that $F$ is a linear function.
In fact, axiom (D) makes $\mathcal{H}$ into a Hermitian space~\cite[Definition~1.1]{holland:soler}: the Hermitian form $\inprod{f}{g}$ of $f,g \in \mathcal{H}$ is $g^\dag \circ f \in \mathcal{I}$; nondegeneracy follows from axiom (E)~\cite[Lemma~2.5]{vicary:complex}.
Hence we can speak about the orthocomplement $V^\perp = \{ f \in \mathcal{H} \mid \forall v \in V \colon \inprod{f}{v}=0 \}$ of a subspace $V \subseteq \mathcal{H}$.
A subspace $V$ is called \emph{closed} when $V^{\perp \perp}=V$.
We now justify using the same notation for this orthocomplementation, which ostensibly differs from that of Lemma~\ref{lem:completeortholattice}.

\begin{lemma}\label{lem:ortholattices}
  The function $p \mapsto p \circ \mathcal{H} = \{ p \circ h \mid h \in \mathcal{H} \}$ is an isomorphism from the ortholattice of projections on $H$ to the ortholattice of closed subspaces of $\mathcal{H}$. 
\end{lemma}
\begin{proof}
  For a projection $p:H \to H$ and a vector $h \in \mathcal{H}$:
  \begin{align*}
    h \in p^\perp \circ \mathcal{H}
    & \iff (\id_H-p) \circ h = h
    \iff p \circ h = 0 \\
    & \iff \forall f \in \mathcal{H} \colon (p \circ f)^\dag \circ h = f^\dag \circ p \circ h = 0 \\
    & \iff h \in (p \circ \mathcal{H})^\perp.
  \end{align*}
  The third equivalence uses that $I$ is a monoidal separator as in axiom (T).
  Thus $p^\perp \circ \mathcal{H} = (p \circ \mathcal{H})^\perp$, and $(p \circ \mathcal{H})^{\perp\perp} = p \circ \mathcal{H}$ is indeed a closed subspace of $\mathcal{H}$.

  Observe that $0_{H,H} \circ \mathcal{H} = \{0_{I,H}\}$, that $\id_H \circ \mathcal{H} = \mathcal{H}$, and that $p \circ \mathcal{H} \subseteq q \circ \mathcal{H}$ when $p \leq q$.
  Hence $p \mapsto p \circ \mathcal{H}$ is a monotone function that respects orthocomplementation. It remains to show that it  is an order embedding and that it is surjective.

  Let $p$ and $q$ be projections on $H$. 
  If $p \circ \mathcal{H} \subseteq q \circ \mathcal{H}$, then $q \circ p \circ f = p \circ f$ for all $f \in \mathcal{H}$, so by axiom (T), $q \circ p = p$, that is, $p \leq q$.
  Therefore, the function $p \mapsto p \circ \mathcal{H}$ is an order embedding, and in particular injective.

  To establish that it is surjective, let $V$ be a closed subspace of $\mathcal{H}$. For a nonzero vector $h \in \mathcal{H}$,  define a projection
  \[
    p_h = (h \circ h^\dag) \cdot (h^\dag  \circ h)\inv : H \to H
  \]
  where the dot denotes the scalar multiplication of~\eqref{eq:scalarmultiplication}. 
  Set $p = \bigvee_{v \in V\setminus\{0\}} p_v$. For each vector $v \in V\setminus\{0\}$ then $p_v \circ v = v$ and so $p \circ v = p \circ p_v  \circ v = p_v \circ v = v$. Thus $V \subseteq p \circ \mathcal{H}$.

  Let $h$ be a nonzero vector in $\mathcal{H}$, and assume that $h \in V^\perp$. Then $p_h \circ p_v = 0$ for all $v \in V\setminus\{0\}$; in other words, $p_v \leq p_h^\perp$. 
  It follows that $p = \bigvee_{v \in V \setminus \{0\}} p_v \leq p_h^\perp$, and so $p_h \circ p =0$. 
  Hence $p_h \circ \mathcal{H} \perp p \circ \mathcal{H}$, and in particular $h \perp p \circ \mathcal{H}$. Therefore any vector $h \in \mathcal{H}$ that is orthogonal to $V$ is also orthogonal to $p \circ \mathcal{H}$. Since both $V$ and $p \circ \mathcal{H}$ are closed subspaces of $\mathcal{H}$, we conclude that $p \circ \mathcal{H} \subseteq V$. Altogether, $p \circ \mathcal{H} = V$, and the function $p \mapsto p \circ \mathcal{H}$ is surjective.
\end{proof}

\section*{Orthogonality}

A Hermitian space $\mathcal{H}$ is called \emph{orthomodular} when $\mathcal{H} = V \oplus V^\perp$ for all closed subspaces $V \subseteq \mathcal{H}$~\cite[Definition~1.2]{holland:soler}.

\begin{lemma}\label{lem:orthomodularspace}
  $\mathcal{H}$ is an orthomodular space.
\end{lemma}
\begin{proof}
  Let $V$ be a closed subspace of $\mathcal{H}$, and let $h$ be a vector in $\mathcal{H}$. Lemma~\ref{lem:ortholattices} provides a projection $p$ on $H$ such that $V = p \circ \mathcal{H}$.
  Trivially $p \circ h \in V$. Also $p^\perp \circ h \in V^\perp$ because for all $v \in V$:
  \[
    v^\dag \circ (p^\perp \circ h) 
    = ((\id_H - p) \circ v)^\dag \circ h 
    = (v - (p \circ v))^\dag \circ h 
    = 0.
  \]
  We conclude that $h = (p \circ h) + (p^\perp \circ h) \in V + V^\perp$ and more generally that $\mathcal{H} = V \oplus V^\perp$. 
\end{proof}

The next step uses orthogonal subsets of $\mathcal{H}$ that can be infinite. This requires some preparation.
If $R$ is a finite set, axiom (B) shows that a coproduct $\bigoplus_R I$ of copies of $I$ indexed by $R$ exists; denote it by $I^R$.
We may assume that $I^R = I$ when $R$ is a singleton.
Write $i_{R,S}: I^R \to I^S$ for the canonical morphism induced by an inclusion $R \subseteq S$ of finite sets. 
Then $i_{S,T} \circ i_{R,S} = i_{R,T}$ for all finite sets $R \subseteq S \subseteq T$.
Moreover, using the matrix calculus for morphisms afforded by axiom (B)~\cite[Section~2.2.4]{heunenvicary}, it is clear that $i_{R,S}$ is a dagger monomorphism.



Now let $A$ be an arbitrary set. 
Write $R \Subset A$ if $R$ is a finite subset of $A$. 
The morphisms $i_{R,S}$ form a diagram of dagger monomorphisms in $\cat{C}$ indexed by the category 
of finite subsets $R \Subset A$ and inclusions between them.
By axiom (C) this diagram has a colimit $I^A$, with universal dagger monomorphisms $i_{R,A} \colon I^R \to I^A$.
(If $A$ is finite, the diagram has a terminal object, so both definitions of $I^A$ agree.)

\begin{proposition}\label{prop:hilbertspace}
  The involutive field $\mathcal{I}$ is isomorphic to $\mathbb{R}$ (with the trivial involution) or $\mathbb{C}$, and $\mathcal{H}$ is a Hilbert space for any object $H$.
\end{proposition}
\begin{proof}
  In a first step, we show that $e_a = i_{\{a\},A} \colon I \to I^A$ form an orthonormal subset of $\cat{C}(I,I^A)$, where $a$ ranges over an arbitrary set $A$. Because $e_a$ is a dagger monomorphism, $\inprod{e_a}{e_a} = e_a^\dag \circ e_a = \id_{I}$. Similarly, for distinct $a, b \in A$:
  \begin{align*}
    \inprod{e_a}{e_b} 
    & = e_a^\dag \circ e_b 
    = i_{\{a\},\{a,b\}}^\dag \circ i_{\{a,b\},A}^\dag \circ i_{\{a,b\},A} \circ i_{\{b\},\{a,b\}}
    = i_{\{a\},\{a,b\}}^\dag \circ i_{\{b\},\{a,b\}} 
    = 0,
  \end{align*}
  because $i_{\{a\},\{a,b\}}: I^{\{a\}} \to I^{\{a,b\}}$ and $i_{\{b\},\{a,b\}}: I^{\{b\}} \to I^{\{a,b\}}$ are inclusions of distinct summands of $I^{\{a,b\}}= I^{\{a\}} \oplus I^{\{b\}}$.

  Thus $\cat{C}(I,I^\NN)$, which is an orthomodular space over $\mathcal{I}$ by Lemma~\ref{lem:orthomodularspace}, has an infinite orthonormal subset. It follows from Sol\`er's Theorem that $\mathcal{I}$ is isomorphic to $\mathbb{R}$ or $\mathbb{C}$, see~\cite{soler} or~\cite[Theorem~1.3]{holland:soler}. It is not isomorphic to $\mathbb{H}$ because it is commutative~\cite{baez:division}.

  Let $H$ be an object of $\cat{C}$. The Hermitian space $\cat{C}(I,H \oplus I^\NN)$ is an orthomodular space by Lemma~\ref{lem:orthomodularspace}. It has an infinite orthonormal subset because the inclusion morphism $I^\NN \to H \oplus I^\NN$ is a dagger monomorphism, and thus postcomposition by this inclusion morphism is a function $\cat{C}(I, I^\NN) \to \cat{C}(I,H \oplus I^\NN)$ that preserves the Hermitian form. Appealing again to Sol\`er's Theorem, we conclude that $\cat{C}(I, H \oplus I^\NN)$ is a Hilbert space over $\mathcal{I}$.

  Similarly, postcomposing with the inclusion morphism $H \to H \oplus I^\NN$ gives a function $\mathcal{H} \to \cat{C}(I,H \oplus I^\NN)$ that preserves the Hermitian form. It follows that the Hermitian form on $\mathcal{H}$ is an inner product. If $\mathcal{H}$ has an infinite orthonormal subset, then it is a Hilbert space over $\mathcal{I}$ by Lemma~\ref{lem:orthomodularspace} and Sol\`er's Theorem. If $\mathcal{H}$ does not have an infinite orthonormal subset, then it is simply a finite-dimensional inner product space over $\mathcal{I}$ and thus also a Hilbert space over $\mathcal{I}$.
\end{proof}

If there exists a morphism $z \colon I \to I$ such that $z^\dag \neq z$, then $\mathcal{I}=\mathbb{C}$, otherwise $\mathcal{I}=\mathbb{R}$. 
This gives one way to distinguish between the complex and real cases.
Another way is to ask that there is a morphism $z \colon I \to I$ with $z \circ z = -1$.
There may be more, for example inspired by symmetries~\cite{mayet:automorphisms}.
We will abbreviate $\cat{Hilb}_{\mathcal{I}}$ to $\cat{Hilb}$ from now on.

\begin{corollary}
  The assignment of $F \colon \mathcal{H} \to \mathcal{K}$ to $f \colon H \to K$ is a functor $\cat{C}(I,-) \colon \cat{C} \to \cat{Hilb}$.
\end{corollary}
\begin{proof}
  If $f : H \to K$ is a morphism in $\cat{C}$, we have to show that the function $F \colon \mathcal{H} \to \mathcal{K}$ is continuous and linear. Linearity follows from  axioms (B) and (T)~\cite[Lemma~2.6 and Proposition~2.23]{heunenvicary}.
  Similarly, postcomposition by $t=f^\dag \circ f \colon H \to H$ is a linear operator $T : \mathcal{H} \to \mathcal{H}$. For all $g,h \in \mathcal{H}$, we find
  $\inprod{h}{T(g)} 
  = h^\dag \circ (t \circ g) 
  = (t \circ h)^\dag \circ g 
  = \inprod{T(h)}{g}$, so $T$ is self-adjoint. 
  It follows from the Hellinger-Toeplitz theorem that $T$ is continuous~\cite[Corollary~III.12]{reedsimon:functionalanalysis}. 
  Because
  \begin{align*}
    \| F(g) \|^2
    = \| f \circ g\|^2 
    = g^\dag \circ t \circ g
    = \inprod{g}{T(g)} 
    \leq \|g\|^2 \|T\|\text,
  \end{align*}
  $F$ has norm at most $\|T\|^{1/2}$ and is therefore continuous.
\end{proof}

\section*{Bases}

We continue the analysis of $\cat{C}(I,I^A)$ by showing that the orthonormal set in the proof of Proposition~\ref{prop:hilbertspace} is in fact an orthonormal basis. Recall that the dimension of a Hilbert space is defined to be the cardinality of any orthonormal basis.

%

\begin{lemma}\label{lem:standardbasis}
  For any set $A$, the dimension of the Hilbert space $\cat{C}(I,I^A)$ is the cardinality~of~$A$.
\end{lemma}
\begin{proof}
  Proposition~\ref{prop:hilbertspace} showed that $e_a = i_{\{a\},A}$, where $a$ ranges over $A$, is an orthonormal subset of $\cat{C}(I,I^A)$. We will show that it is an orthonormal basis.
  Suppose that a vector $f : I \to I^A$ is orthogonal to $e_a$ for all $a \in A$. Given $a \in R \Subset A$, we find $f^\dag \circ i_{R,A} \circ i_{\{a\},R} = f^\dag \circ e_a = 0$.
  Because $i_{\{a\},R}: I \to I^R$ are the coprojections of the coproduct $I^R$, we have $f^\dag \circ i_{R,A} = 0$ for all $R \Subset A$. The morphism $f^\dag: I^A \to I$ has a dagger monomorphism $\ker(f^\dag) : K \to I^A$ as kernel. This gives dagger monomorphisms $g_R: I^R \to K$ such that $ker(f^\dag) \circ g_R = i_{R,A}$, for $R \Subset A$.
  These form a cocone on the diagram of dagger monomorphisms $i_{R,S}$ for $R,S\Subset A$: for if $R \Subset S \Subset A$, then
  \[
    ker(f^\dag) \circ g_S \circ i_{R,S} 
    = i_{S,A} \circ i_{R,S} 
    = i_{R,A} 
    = ker(f^\dag) \circ g_R\text,
  \]
  and because $\ker(f^\dag)$ is a dagger monomorphism, we have that $g_S \circ i_{S,R} = g_R$.
  \[\begin{tikzcd}[column sep=7em, row sep=4em]
    {I^{\{a\}}} & {I^R} & \cdots \\
    K & {I^A} & I
    \arrow["{i_{\{a\},R}}"{description}, tail, from=1-1, to=1-2]
    \arrow["{i_{R,S}}"{description}, tail, from=1-2, to=1-3]
    \arrow["{i_{R,A}}"{description}, tail, from=1-2, to=2-2]
    \arrow["{f^\dag}"{description}, from=2-2, to=2-3]
    \arrow["{\ker(f^\dag)}"{description}, from=2-1, to=2-2, tail]
    \arrow["g"{description}, out=-120, in=-60, looseness=.9, dashed, from=2-2, to=2-1, tail]
    \arrow["{g_{\{a\}}}"{description}, tail, from=1-1, to=2-1]
    \arrow["{g_R}"{description}, pos=.75, tail, from=1-2, to=2-1]
    \arrow["{e_a}"{description}, pos=.25, tail, from=1-1, to=2-2, crossing over]
    \arrow["0"{description}, from=1-3, to=2-3]
  \end{tikzcd}\]
  Therefore there exists a unique mediating morphism $g: I^A \to K$ with $g \circ i_{R,A}  = g_R$ for all $R \Subset A$. 
  Thus, $ker(f^\dag) \circ g \circ i_{R,A} = ker(f^\dag) \circ g_R = i_{R,A}$ for all $R \Subset A$, and by the universal property of the colimit $I^A$, we conclude $\ker(f^\dag) \circ g = \id_{I^A}$. It follows that $ker(f^\dag)$ is an epimorphism, and being a dagger monomorphism, it is a dagger isomorphism. Therefore, $f^\dag = 0$.

  Overall, the zero vector is the only one in the Hilbert space $\cat{C}(I,I^A)$ that is orthogonal to each vector $e_a$, for $a \in A$. Thus, $\{e_a\}_{a \in A}$ is an orthonormal basis of $\cat{C}(I,I^A)$.
\end{proof}



\section*{Equivalence}

We are now ready to prove the main result. A functor is a dagger equivalence when it is an equivalence of categories that preserves the dagger of morphisms.

\begin{theorem}\label{thm:equivalence}
  The functor $\cat{C}(I,-)\:\cat{C} \to \cat{Hilb}_{\mathcal{I}}$ is an equivalence of dagger categories, where $\mathcal{I}$ is either $\mathbb{R}$ or $\mathbb{C}$.
\end{theorem}
\begin{proof}
  Let $f : H \to K$ be a morphism in $\cat{C}$.
  For all vectors $h \in \mathcal{H}$ and $k \in \mathcal{K}$ axiom (D) shows:
  \[
    \inprod{f^\dag \circ k}{h} 
    = (f^\dag \circ k)^\dag \circ h 
    = k^\dag \circ f \circ h 
    = \inprod{k}{f \circ h}.
  \]
  Thus $\cat{C}(I, f^\dag) = \cat{C}(I, f)^\dag$, that is, $\cat{C}(I,-)$ is a dagger functor. 

  If morphisms $f,g : H \to K$ satisfy $\cat{C}(I,f) = \cat{C}(I,g)$, then $f \circ h = g\circ h$ for all vectors $h \in \mathcal{H}$, so $f = g$ by axiom (T). Therefore, $\cat{C}(I, -)$ is faithful.

  Assume $\dim(\mathcal{H}) \leq \dim(\mathcal{K})$, and let $U \colon \mathcal{H} \to \mathcal{K}$ be a linear isometry. 
  Choose an orthonormal basis $\{h_a\}_{a \in A}$ for the Hilbert space $\mathcal{H}$. Because $U$ is an isometry, the vectors $k_a = U(h_a)$ form an orthonormal subset of the Hilbert space $\mathcal{K}$. 

  Extending the dagger monomorphisms $h_a : I \to H$ and $k_a : I \to K$ for finite subsets $R \Subset A$ to dagger monomorphisms $h_R : I^R \to H$ and $k_R : I^R \to K$ using axiom (B) gives two cocones on the diagram $i_{R,S}$. Thus by axiom (C) there are unique mediating dagger monomorphisms $h_A : I^A \to H$ and $k_A : I^A \to K$ satisfying $h_A \circ e_a = h_a$ and $k_A \circ e_a = k_a$ for all $a \in A$, where $e_a = i_{\{a\},A}$ as before.

  Define $u = k_A \circ h_A^\dag : H \to K$. For each $a \in A$ then:
  \begin{align*}
    \cat{C}(I, u)(h_a) 
    & = k_A \circ h_A^\dag \circ h_a 
    = k_A \circ h_A^\dag \circ h_A \circ e_a
    = k_A \circ e_a 
    = k_a 
    = U(h_a).
  \end{align*}
  Because $\{h_a\}_{a \in A}$ is an orthonormal basis of the Hilbert space $\mathcal{H}$, we conclude that $\cat{C}(I, u) =  U$.

  Any continous linear function $\mathcal{H} \to \mathcal{K}$ is a linear combination of isometric ones, because we have assumed that $\dim(\mathcal{H}) \leq \dim(\mathcal{K})$. 
  It follows that each continuous linear function $F:\mathcal{H} \to \mathcal{K}$ is the image of a linear combination of morphisms $H \to K$ under the functor $\cat{C}(I,-)$, which is linear on morphisms by axiom (B).
  Therefore, there exists a morphism $f: H \to K$ such that $\cat{C}(I,f) = F$. 
  Combining this fact for $\dim(\mathcal{H}) \leq \dim(\mathcal{K})$ with the fact that the functor $\cat{C}(I,-)$ preserves daggers, we conclude that the functor is full.

  Finally, the functor $\cat{C}(I,-)$ is essentially surjective by Lemma~\ref{lem:standardbasis}. Being a full, faithful and essentially surjective dagger functor, it is an equivalence of dagger categories.
\end{proof}

Because equivalences preserve limits, it follows directly that the functor $\cat{C}(I,-) : \cat{C} \to \cat{Hilb}$ preserves all the structure concerned in axioms (D), (B), (E), and (K).
It also follows that the functor restricts to a full and faithful functor between the wide subcategories of dagger monomorphisms; because in $\cat{Hilb}$ any two isomorphic objects are unitarily isomorphic~\cite[8.3, 8.4]{heunenkarvonen:limitsindaggercategories}, the restricted functor is still an equivalence, and hence also the structure concerned in axiom (C) is preserved.

\section*{Tensor products}

Finally, we prove that the equivalence $\cat{C}(I,-) : \cat{C} \to \cat{Hilb}$ also preserves the tensor products of axiom (T).

For objects $H$ and $K$, the function $\mathcal{H} \times \mathcal{K} \to \cat{C}(I, H \otimes K)$ given by $(h,k) \mapsto h \otimes k$ is bilinear by axioms (B) and (T), see~\cite[Lemma~3.6, Corollary~3.20, and Lemma~3.22]{heunenvicary}. It is also bounded as a bilinear map because 
\begin{align*}
  \|h \otimes k\|^2 
  & = \inprod{h \otimes k}{h \otimes k} 
  = (h \otimes k)^\dag \circ (h \otimes k) \\
  & = (h^\dag \circ h) \otimes (k^\dag \circ k) 
  = \inprod{h}{h} \cdot \inprod{k}{k} 
  = \|h\|^2 \cdot \|k\|^2
\end{align*}
where we suppressed the isomorphism $I \otimes I \simeq I$.
This defines a continuous linear function:
\begin{equation}\label{eq:tensor}
  \begin{split}
    M_{H,K} : \mathcal{H} \otimes \mathcal{K} &\to \cat{C}(I, H \otimes K) \\
    h \otimes k & \mapsto h \otimes k.
  \end{split}
\end{equation}
Here the tensor product on the left-hand side is that of Hilbert spaces, whereas the tensor product on the right-hand side is the monoidal structure of $\cat{C}$.

\begin{lemma}\label{lem:tensor}
  The continuous linear functions of~\eqref{eq:tensor} are unitary and form the components of a natural transformation from the functor $\cat{C}(I,-) \otimes \cat{C}(I,-) : \cat{C} \times \cat{C} \to \cat{Hilb}$ to the functor $ \cat{C}(I, - \otimes -) : \cat{C} \times \cat{C} \to \cat{Hilb}$.
\end{lemma}
\begin{proof}
  If $h,h' \in \mathcal{H}$ and $k,k' \in \mathcal{K}$, then the inner product
  $\inprod{h \otimes k}{h' \otimes k'}$ in $\cat{C}(I,H \otimes K)$ is equal to $\inprod{h}{h'} \cdot \inprod{k}{k'}$.
  Hence $M_{H,K}$ preserves inner products between elementary tensors, and
  linearity and continuity then imply that $M_{H,K}$ is an isometry.
  So its range is a closed subspace of $\cat{C}(I, H \otimes K)$. Suppose there were a nonzero vector $f$ orthogonal to $h \otimes k$ for all $h \in \mathcal{H}$ and $k \in \mathcal{K}$. In other words, $f^\dag \circ (h \otimes k) = \inprod{f}{h \otimes k} = 0$. It would then follow from axiom (T) that $f^\dag = 0$, contradicting our choice of $f$. Therefore, $M_{H,K}$ is surjective, and hence unitary.

  If $f:H_1 \to H_2$ and $g:K_1 \to K_2$ are morphisms of $\cat{C}$, then the following diagram commutes:
  \[\begin{tikzcd}[row sep = large,column sep = huge]
    \cat{C}(I,H_1) \otimes \cat{C}(I,K_1)
    \arrow{d}[swap]{\cat{C}(I,f) \otimes \cat{C}(I,g)}
    \arrow{r}{M_{H_1,K_1}}
    &
    \cat{C}(I, H_1 \otimes K_1)
    \arrow{d}{\cat{C}(I, f \otimes g)}
    \\
    \cat{C}(I,H_2) \otimes \cat{C}(I,K_2)
    \arrow{r}[swap]{M_{H_2,K_2}}
    &
    \cat{C}(I, H_2 \otimes K_2)
  \end{tikzcd}\]
  Indeed, for all $h \in \mathcal{H}_1$ and $k \in \mathcal{K}_1$:
  \begin{align*}
    \cat{C}(I, f \otimes g)&(M_{H_1, K_1}(h \otimes k))
    = \cat{C}(I, f \otimes g)(h \otimes k)
    = (f \circ h) \otimes (g \circ k) \\
    & = \cat{C}(I, f)(h) \otimes \cat{C}(I, g)(k)
    = M_{H_2, K_2}(\cat{C}(I, f)(h) \otimes \cat{C}(I, g)(k)) \\
    & = M_{H_2, K_2}((\cat{C}(I, f) \otimes \cat{C}(I, g))(h \otimes k)).
  \end{align*}
  Therefore, the unitary linear functions $M_{H,K}$ are natural in $H$ and $K$.
\end{proof}

A (dagger) functor is monoidal when it preserves (the dagger and) the tensor product up to a natural (dagger) isomorphism that respects the (dagger) isomorphisms $(H \otimes K) \otimes L \simeq H \otimes (K \otimes L)$ and that preserves the tensor unit up to (dagger) isomorphism~\cite[Section 2.4]{etingofgelakinikshychostrik:tensorcategories}. 
If a monoidal functor is an equivalence, then its adjoint is automatically also monoidal.

\begin{theorem}\label{thm:tensor}
  The functor $\cat{C}(I,-) : \cat{C} \to \cat{Hilb}$ is a monoidal dagger equivalence.
\end{theorem}
\begin{proof}
  After Theorem~\ref{thm:equivalence}, it suffices to show that the functor $\cat{C}(I,-)$ together with the natural transformation $M$ of Lemma~\ref{lem:tensor} and the identity operator $E$ on $\mathcal{I}$ is a monoidal functor. By construction, the Hilbert space $\cat{C}(I,I)$ is just the field $\mathcal{I}$ considered as a Hilbert space over itself, and so $E$ is a dagger isomorphism from the value of the functor $\cat{C}(I,-)$ at $I$ to the tensor unit of $\cat{Hilb}_{\mathcal{I}}$.


  Writing $a_{H,K,L} : (H \otimes K) \otimes L \to H \otimes (K \otimes L)$ and $A_{\mathcal{H},\mathcal{K},\mathcal{L}} : (\mathcal{H} \otimes \mathcal{K}) \otimes \mathcal{L} \to \mathcal{H} \otimes (\mathcal{K} \otimes \mathcal{L})$ for the associators of $\cat{C}$ and $\cat{Hilb}$, the diagram
  \[\begin{tikzcd}[column sep = 6em]
    (\mathcal{H} \otimes \mathcal{K}) \otimes \mathcal{L}
    \arrow{r}{A_{\mathcal{H}, \mathcal{K}, \mathcal{L}}}
    \arrow{d}[swap]{M_{H,K} \otimes \id_{\mathcal{L}}}
    &
    \mathcal{H} \otimes (\mathcal{K} \otimes \mathcal{L})
    \arrow{d}{\id_{\mathcal{H}} \otimes M_{K,L}}
    \\
    \cat{C}(I, H \otimes K) \otimes \mathcal{L}
    \arrow{d}[swap]{M_{H \otimes K, L}}
    &
    \mathcal{H} \otimes \cat{C}(I, K \otimes L)
    \arrow{d}{M_{H, K \otimes L}}
    \\
    \cat{C}(I, (H \otimes K) \otimes L)
    \arrow{r}[swap]{\cat{C}(I,a_{H,K,L})}
    &
    \cat{C}(I, H \otimes (K \otimes L))
  \end{tikzcd}\]
  commutes because for all $h \in \mathcal{H}$, $k \in \mathcal{K}$ and $l \in \mathcal{L}$:
  \begin{align*}
  M_{H, K \otimes L} & (\id \otimes M_{K,L}(A_{\mathcal{H} \otimes \mathcal{K}, \mathcal{L}}((h \otimes k) \otimes l)))
  = M_{H, K \otimes L} (\id \otimes M_{K,L}(h \otimes (k \otimes l))) \\
  & = h \otimes (k \otimes l) 
    = a_{H,K,L} \circ ((h \otimes k) \otimes l)
  = \cat{C}(I,a_{H,K,L})((h \otimes k) \otimes l) \\
  & = \cat{C}(I,a_{H,K,L})(M_{H \otimes K, L}((M_{H,K} \otimes \id) ((h \otimes k) \otimes l))).
  \end{align*}
  Thus $\cat{C}(I,-)$ is a monoidal functor.
\end{proof}

If the category $\cat{C}$ is equipped with a dagger braiding, then the functor $\cat{C}(I, -)$ is dagger braided by a similar argument. This means that $\cat{C}(I, -)$ respects the dagger isomorphisms $H \otimes K \simeq K \otimes H$. Furthermore, the braiding on $\cat{C}$ is symmetric as a consequence of the assumption that $I$ is a monoidal separator in axiom (T).

Recall that $H$ and $H^*$ are \emph{dagger dual} objects when there is a morphism $f \colon I \to H^* \otimes H$ making the following diagram commute~\cite[Definition~3.46]{heunenvicary}:
\[\begin{tikzcd}[column sep = .5em]
  {H \otimes I} & H & {I \otimes H} \\
  {H \otimes (H^* \otimes H)} & {(H \otimes H^*) \otimes H} & {(H^* \otimes H) \otimes H}
  \arrow["{\id_H \otimes f}"', from=1-1, to=2-1]
  \arrow["\simeq"{description}, draw=none, from=2-1, to=2-2]
  \arrow["\simeq"{description}, draw=none, from=2-2, to=2-3]
  \arrow["{f^\dag \otimes \id_H}"', from=2-3, to=1-3]
  \arrow["\simeq"{description}, draw=none, from=1-2, to=1-3]
  \arrow["\simeq"{description}, draw=none, from=1-1, to=1-2]
\end{tikzcd}\]
In $\cat{Hilb}$, an object has a dagger dual if and only if it is finite-dimensional~\cite[Corollary~3.65]{heunenvicary}.
Because monoidal dagger functors preserve dagger dual objects~\cite[Theorem~3.14]{heunenvicary}, it follows from Theorem~\ref{thm:tensor} that the finite-dimensional Hilbert spaces can be categorically axiomatised within $\cat{C}$ as the dagger dual objects.

\bibliographystyle{plain}
\bibliography{bibliography}

\begin{thebibliography}{10}

\bibitem{abramskycoecke:protocols}
S.~Abramsky and B.~Coecke.
\newblock A categorical semantics of quantum protocols.
\newblock {\em Proceedings of the 19th Annual IEEE Symposium on Logic in
  Computer Science, 2004}, pages 415--425, 2004.

\bibitem{amemiyaaraki:piron}
I.~Amemiya and H.~Araki.
\newblock A remark on {P}iron's paper.
\newblock {\em Publications of the Research Institute for Mathematical
  Sciences}, 2(3):423--427, 1966.

\bibitem{baez:division}
J.~C. Baez.
\newblock Division algebras and quantum theory.
\newblock {\em Foundations of Physics}, 42:819--855, 2012.

\bibitem{etingofgelakinikshychostrik:tensorcategories}
P.~Etingof, S.~Gelaki, D.~Nikshych, and V.~Ostrik.
\newblock {\em Tensor Categories}.
\newblock American Mathematical Society, 2015.

\bibitem{ghezlimaroberts:wstarcategories}
P.~Ghez, R.~Lima, and J.~E. Roberts.
\newblock {$W^*$}-categories.
\newblock {\em Pacific Journal of Mathematics}, 120:79--109, 1985.

\bibitem{grinbaum:reconstruction}
A.~Grinbaum.
\newblock Reconstruction of quantum theory.
\newblock {\em British Journal for the Philosophy of Science}, 58(3):387--408,
  2007.

\bibitem{gross:hilbertlattices}
H.~Gross.
\newblock Hilbert lattices: new results and unsolved problems.
\newblock {\em Foundations of Physics}, 20:529--559, 1990.

\bibitem{hardy:fiveaxioms}
L.~Hardy.
\newblock Quantum theory from five reasonable axioms.
\newblock arXiv:0101012, 2001.

\bibitem{heunen:embedding}
C.~Heunen.
\newblock An embedding theorem for {H}ilbert categories.
\newblock {\em Theory and Applications of Categories}, 13:321--344, 2009.

\bibitem{heunenjacobs:daggerkernelcategories}
C.~Heunen and B.~Jacobs.
\newblock Quantum logic in dagger kernel categories.
\newblock {\em Order}, 27:177--212, 2010.

\bibitem{heunenkarvonen:limitsindaggercategories}
C.~Heunen and M.~Karvonen.
\newblock Limits in dagger categories.
\newblock {\em Theory and Applications of Categories}, 34(18):468--513, 2019.

\bibitem{heunenvicary}
C.~Heunen and J.~Vicary.
\newblock {\em Categories for Quantum Theory}.
\newblock Oxford University Press, 2019.

\bibitem{jauch:foundations}
J.~M. Jauch.
\newblock {\em Foundations of Quantum Mechanics}.
\newblock Addison-Wesley, 1968.

\bibitem{holland:soler}
S.~S.~Holland Jr.
\newblock Orthomodularity in infinite dimensions; a theorem of {M}.
  {S}ol\`{e}r.
\newblock {\em Bulletin of the American Mathematical Society}, 32:205--234,
  1995.

\bibitem{kaplansky:forms}
I.~Kaplansky.
\newblock Forms in infinite-dimensional spaces.
\newblock {\em Anais da {A}cademia {B}rasileira de {C}i{\^e}ncias}, 22:1--17,
  1950.

\bibitem{kapustin:exact}
A.~Kapustin.
\newblock Is quantum mechanics exact?
\newblock {\em Journal of Mathematical Physics}, 54:062107, 2013.

\bibitem{lawvere:etcs}
F.~W. Lawvere.
\newblock An elementary theory of the category of sets.
\newblock {\em Proceedings of the National Academy of Sciences},
  52(6):1506--1511, 1964.

\bibitem{maclane:duality}
S.~{Mac Lane}.
\newblock Duality for groups.
\newblock {\em Bulletin of the American Mathematical Society}, 56(6):485--516,
  1950.

\bibitem{maclane:categories}
S.~{Mac Lane}.
\newblock {\em Categories for the Working Mathematician}.
\newblock Springer, 2nd edition, 1971.

\bibitem{maedamaeda}
F.~Maeda and S.~Maeda.
\newblock {\em Theory of symmetric lattices}.
\newblock Springer, 1970.

\bibitem{mayet:automorphisms}
R.~Mayet.
\newblock Some characterizations of the underlying division ring of a {H}ilbert
  lattice by automorphisms.
\newblock {\em International Journal of Theoretical Physics}, 37(1):109--114,
  1998.

\bibitem{mitchell:embedding}
B.~Mitchell.
\newblock The full embedding theorem.
\newblock {\em Johns Hopkins University Press}, 1964.

\bibitem{murrayneumann:ringsofoperators}
F.~J. Murray and J.~{von Neumann}.
\newblock On rings of operators.
\newblock {\em Annals of Mathematics}, 37(1):116--229, 1936.

\bibitem{piron:conjecture}
C.~Piron.
\newblock Axiomatique quantique.
\newblock {\em Helvetica Physica Acta}, 37:439--468, 1964.

\bibitem{piron:foundations}
C.~Piron.
\newblock {\em Foundations of Quantum Physics}.
\newblock W. A. Benjamin, 1976.

\bibitem{redei:vonneumannhilbert}
M.~R{\'e}dei.
\newblock Why {J}ohn von {N}eumann did not like the {H}ilbert space formalism
  of quantum mechanics (and what he liked instead).
\newblock {\em Studies in the History and Philosophy of Modern Physics},
  27(4):493--510, 1996.

\bibitem{reedsimon:functionalanalysis}
M.~Reed and B.~Simon.
\newblock {\em Methods of modern mathematical physics {I}: functional
  analysis}.
\newblock Academic Press, 1980.

\bibitem{selinger:cpm}
P.~Selinger.
\newblock Dagger compact closed categories and completely positive maps.
\newblock In {\em Quantum Physics and Logic}, volume 170 of {\em Electronic
  Notes in Theoretical Computer Science}, pages 139--163, 2007.

\bibitem{soler}
M.~P. Sol\`{e}r.
\newblock Characterization of {H}ilbert spaces by orthomodular spaces.
\newblock {\em Communications in Algebra}, 1:219--243, 1995.

\bibitem{vetterlein:orthogonalityspaces}
T.~Vetterlein.
\newblock Orthogonality spaces arising from infinite-dimensional complex
  {H}ilbert spaces.
\newblock {\em International Journal of Theoretical Physics}, 60(2):727--738,
  2021.

\bibitem{vicary:complex}
J.~Vicary.
\newblock Completeness of $\dagger$-categories and the complex numbers.
\newblock {\em Journal of Mathematical Physics}, page 082104, 2011.

\bibitem{vonneumann:foundations}
J.~{von Neumann}.
\newblock {\em Mathematical Foundations of Quantum Mechanics}.
\newblock Princeton University Press, 1955.

\bibitem{wilbur:standardquantumlogic}
W.~J. Wilbur.
\newblock On characterizing the standard quantum logics.
\newblock {\em Transactions of the American Mathematical Society},
  233:265--282, 1977.

\end{thebibliography}

\end{document}